\documentclass{article}

\usepackage{url}

\usepackage[normalem]{ulem}
\usepackage{comment}
\usepackage{xcolor}
\usepackage{graphicx}
\usepackage{hyperref}
\usepackage{amsfonts}
\usepackage{amsmath}
\usepackage{amsthm}
\usepackage[numbers]{natbib}
\usepackage[utf8]{inputenc}
\usepackage[T1]{fontenc}
\usepackage[affil-it]{authblk}
\newtheorem{lmm}{Lemma}
\newtheorem{thrm}{Theorem}

\newtheorem{prpstn}{Proposition}
\newtheorem{rmrk}{Remark}
\newtheorem{assumption}{Assumption}
\newtheorem{dfntn}{Definition}
\newtheorem{xmpl}{Example}
\def\indicator#1{\mathbf{1}_{\{#1\}}}
\newcommand\var{\operatorname{\mathbf{var}}}

\newcommand\E{\operatorname{\mathbf{E}}}

\newcommand{\xR}{\mathbb{R}} 
\newcommand{\xN}{\mathbb{N}} 
\newcommand{\xC}{\mathbb{C}} 
\newcommand{\p}{\mathbb{P}}

\usepackage{url}
\makeatletter
\newcommand*{\bigcdot}{}% Check if undefined
\DeclareRobustCommand*{\bigcdot}{%
  \mathbin{\mathpalette\bigcdot@{}}%
}
\newcommand*{\bigcdot@scalefactor}{.5}
\newcommand*{\bigcdot@widthfactor}{1.15}
\newcommand*{\bigcdot@}[2]{%
  \sbox0{$#1\vcenter{}$}% math axis
  \sbox2{$#1\cdot\m@th$}%
  \hbox to \bigcdot@widthfactor\wd2{%
    \hfil
    \raise\ht0\hbox{%
      \scalebox{\bigcdot@scalefactor}{%
        \lower\ht0\hbox{$#1\bullet\m@th$}%
      }%
    }%
    \hfil
  }%
}

\begin{document}

\title{Singular control with state-dependent costs\\for Lévy processes}
%\author{Ernesto Mordecki, Nora Muler and Facundo Oliú}

\author{Ernesto Mordecki\footnote{Centro de Matemática, Facultad de Ciencias, Universidad de la República, Montevideo, Uruguay, email: mordecki@cmat.edu.uy} ,  Nora Muler\footnote{Departamento de Matemática y Estadística, Universidad Torcuato di Tella, Buenos Aires, Argentina} \ and Facundo Oliú\footnote{Ingeniería Forestal, Centro Universitario de Tacuarembó, Universidad de la República, Tacuarembó, Uruguay} \footnote{Department of Mathematics, Christian-Albrechts-Universität zu Kiel, Germany.}}

\date{\today}

\maketitle
\begin{abstract}
We study a discounted singular stochastic control problem driven by a general Lévy process, where the objective is to minimize a cost functional composed of a running cost and a control cost that depends on the current state of the process. 
We first establish a Hamilton–Jacobi–Bellman (HJB)-type verification theorem providing sufficient conditions under which a reflecting barrier strategy is optimal and characterizing the value function.
Our main contribution is to connect this control problem with an associated optimal stopping problem: we prove that the optimal reflection threshold coincides with the optimal stopping boundary of the auxiliary problem. 
This connection allows us to characterize the optimal strategy through probabilistic tools and leads to explicit or semi-explicit solutions in several relevant cases. 
We illustrate the results with several examples, including an application to pollution abatement.
\end{abstract}
\vspace{1em}  

\medskip 


\section{Introduction}
%%%%%%%%%%%%%%%%%%%%%%%%%%%%%%%%%%%%%%%%%%%%%%%%%%%%%%%%%%
%%%%%%%%  SINGULAR STOCHASTIC CONTROL PROBLEMS
%%%%%%%%%%%%%%%%%%%%%%%%%%%%%%%%%%%%%%%%%%%%%%%%%%%%%%%%%%
Singular stochastic control problems arise in many areas of applied probability. 
Representative works include: in stochastic finance and operations research, 
\cite{KAR,HATA,OKSU}; 
in insurance mathematics, 
\cite{APP,AZMU,PAU,SCH}; 
and for general stochastic control theory, 
\cite{BELI,FLESO}. 
%%%%%%%%%%%%%%%%%%%%%%%%%%%%%%%%%%%%%%%%%%%%%%%%%%%%%%%%%%
%%%%%%%%  DIFFERENT MODELS
%%%%%%%%%%%%%%%%%%%%%%%%%%%%%%%%%%%%%%%%%%%%%%%%%%%%%%%%%%
However, most of the existing literature focuses on diffusion processes or on one-sided L\'evy models (with compound Poisson processes with one-sided jumps as a particular case).
The general L\'evy process setting is significantly more challenging than the spectrally one-sided case because the scale function approach is no longer available. 
Furthermore, the one-sided assumption, while convenient for analytical tractability, is often unrealistic in practice. 
For example, in mathematical finance, asset prices are known to empirically exhibit both upward and downward jumps, and in management or environmental contexts, system dynamics may involve shocks in both directions.
%%%%%%%%%%%%%%%%%%%%%%%%%%%%%%%%%%%%%%%%%%%%%%%%%%%%%%%%%%
%%%%%%%%  OUR PAPER
%%%%%%%%%%%%%%%%%%%%%%%%%%%%%%%%%%%%%%%%%%%%%%%%%%%%%%%%%%

In the present paper, we consider a general L\'evy process and aim to minimize a cost functional that depends on both the trajectory of the controlled state process and the applied control. 
More specifically, our results extend previous work of  \cite{NOBAYAMAZAKI}, 
which analyzed the case in which the control cost is proportional to the amount of control and established the optimality of barrier-type strategies in singular control problems for general L\'evy processes with convex running costs.
In our case, in contrast, the marginal cost of control explicitly depends on the current state of the controlled process, and hence the total control cost can become non-linear. 
This state dependence introduces substantial additional technical challenges due to the loss of linearity in the marginal cost. 
Furthermore, our approach allows the consideration of exponentially bounded 
running cost functions (instead of polynomially bounded), 
which are common in mathematical finance.
The extension is particularly relevant in applications where intervention costs vary with the system's state, such as resource extraction, pollution abatement, or energy regulation. 
A similar state-dependent structure was recently studied by  \cite{JMP}, 
who addressed the maximization of cumulative rewards under a spectrally negative L\'evy process. 
%%%%%%%%%%%%%%%%%%%%%%%%%%%%%%%%%%%%%%%%%%%%%%%%%%%%%%%%%%
%%%%%%%%  MAIN RESULTS
%%%%%%%%%%%%%%%%%%%%%%%%%%%%%%%%%%%%%%%%%%%%%%%%%%%%%%%%%%

The first contribution of the paper is Theorem \ref{T:Verification}, a Hamilton--Jacobi--Bellman (HJB)-type verification theorem that gives sufficient conditions under which the optimal policy is a reflecting barrier strategy. 
More precisely, we show that the existence of a function $u$ satisfying a system of HJB inequalities, with a critical threshold $x^*$ separating the two regimes, implies that the reflecting control at $x^*$ is optimal.

Our second and main contribution is Theorem \ref{T:DISCOUNTEDPROBLEMSOLUTION}, which establishes a connection between the control problem and an associated optimal stopping problem (OSP): the critical threshold $x^*$ of Theorem \ref{T:Verification} coincides with the optimal stopping boundary of the OSP, and the function $u$ of Theorem \ref{T:Verification} is a primitive of its value function.

The resulting method of solution is structurally different from the one developed in \cite{NOBAYAMAZAKI} for the constant control cost case.
Furthermore, the connection between control and stopping allows explicit or semi-explicit solutions in several cases, particularly for polynomial and
exponential cost functions, through averaging representations based on the fluctuation theory for L\'{e}vy processes. 
To summarize, our results generalize previous work in three directions: from constant to state-dependent control costs, from polynomial to exponentially bounded running costs, and by clarifying how state dependence modifies the optimal threshold and value function while preserving analytical tractability.

%%%%%%%%%%%%%%%%%%%%%%%%%%%%%%%%%%%%%%%%%%%%%%%%%%%%%%%%%%
%%%%%%%%  CONNECTIONS CONTROL OSP
%%%%%%%%%%%%%%%%%%%%%%%%%%%%%%%%%%%%%%%%%%%%%%%%%%%%%%%%%%
The connection between control and optimal stopping problems has been known since the seminal work of \cite{KAR}, who studied the problem of tracking a Brownian motion by a process of bounded variation under various cost functionals. 
A more recent treatment of this connection in the diffusion setting is given in \cite{CDF}, 
where further references can be found. 

On the other hand, solutions to optimal stopping problems for L\'evy processes have been known since the paper by \cite{MOR}, with further developments summarized in the book by \cite{KRI}. 
The idea is to express the solution of the OSP through averaging functions involving the supremum and/or infimum of the L\'evy process stopped at an independent exponential time.
When the distribution of these random variables is explicit --typically computed via the Wiener--Hopf factorization--they yield the explicit solution of the OSP. The connection presented here allows us to apply such results to control problems.

It is worth distinguishing the singular control problem considered here
from the notion of stochastic singular systems; see, for instance,
Li, Nie, and Wu \cite{LiNieWu}. We use the term ``singular control'' in
its classical stochastic-control sense, whereas in the latter setting
singularity concerns the structure of the stochastic state equation
rather than the nature of the control.

%%%%%%%%%%%%%%%%%%%%%%%%%%%%%%%%%%%%%%%%%%%%%%%%%%%%%%%%%%
%%%%%%%%  ROAD MAP OF THE PAPER
%%%%%%%%%%%%%%%%%%%%%%%%%%%%%%%%%%%%%%%%%%%%%%%%%%%%%%%%%%
%\subsection{Roadmap}

The rest of the paper is organized as follows. 
%%%%%%%%%%%%
Section 2 introduces the framework and presents the main results, Theorems \ref{T:Verification} and \ref{T:DISCOUNTEDPROBLEMSOLUTION}.
%%%%%%%%%%%%%
Section 3 contains the proofs of these results.
%%%%%%%%%%%%%%
Section 4 provides a collection of examples. In some cases, the solution can be expressed solely in terms of the exponential moments of the infimum of the process stopped at an independent exponential time. In particular, we include an application to pollution abatement.
\section{Main results}%\label{S:framework}

\subsection{Preliminaries}\label{subsection:preliminaries}

Let $X=\{X_t\}_{t\geq 0}$ be a non-trivial L\'evy process with finite mean, 
defined on a stochastic basis 
${\cal B}=(\Omega, {\cal F}, {\bf F}=({\cal F}_t)_{t\geq 0^-}, \p_x)$ starting from $X_{0^-}=x$.
Assume that the filtration is right-continuous and complete (see Definition 1.3 in \cite{JJAS}).
Denote by $\E_x$ the expected value associated to the probability measure $\p_x$, 
and let $\E= \E_0$ and $\p=\p_0$.   
The L\'evy--Khintchine formula characterizes the law of the process through its characteristic exponent
$$
\phi (z)= \log \E( e^{z X_1}), \qquad z=i\theta \in i\xR,
$$
which in our setting takes the form
\begin{equation*}
\phi(z)= \mu z+\frac{\sigma^2}{2}z^2+\int_{\xR}\left(e^{z y}-1-z y\right)\Pi(dy),
\end{equation*}
where $\mu=\E(X_1)\in\xR$, $\sigma\geq 0$ and $\Pi(dy)$ is a non-negative measure (the \emph{jump measure}) satisfying $\int_{\xR}(y^2\wedge |y|)\Pi(dy)<\infty$, which corresponds to the finite-mean assumption on $X$. 
This L\'evy process, being a special semi-martingale (see II.2.29 in \cite{JJAS}), 
can be written as the sum of a deterministic drift, an independent Brownian component, and a pure-jump martingale:
\begin{equation*}
X_t= X_0 +\mu t +\sigma W_t + \int_{[0,t]\times\xR} y\, \tilde N(ds,dy),
\end{equation*}
where $\tilde N(ds,dy)=N(ds,dy)-ds\Pi(dy)$ is a compensated Poisson random measure, 
$N(ds,dy)$ being the jump measure constructed from $X$ (see II.1.16 in \cite{JJAS}),
and $\lbrace W_t \rbrace_{t \geq 0}$ is an independent Brownian motion. 
The infinitesimal generator of the process $X$ is given by
\begin{equation*}
\mathcal{L}u (x)=\lim_{t \to 0^+} \frac{\displaystyle \E_x (u(X_t))-u(x)  }{\displaystyle t} ,
\end{equation*}
which is defined for every function $u\colon\xR\to\xR$ for which the limit exists at every $x \in \xR$; in that case, we say $u$ belongs to the domain of the infinitesimal generator.
For general references on L\'evy processes see 
\cite{B}, \cite{KIS}, and \cite{KRI}.

We consider the following discounted minimization problem starting from $X_{0^-} = x$, where the controlled process $X^D = X - D$ and the class of admissible controls are formally introduced in the next subsection:
\begin{multline*}
   \inf_{D }\E_x \left[
 \int_{0^+}^{\infty}e^{-\delta s} (f(X_{s^-}^{D})ds + c \left(X^{D}_{s^-}\right) dD^c_s)  \right. \\ \left.
    +   \sum_{0^- < s < \infty}e^{-\delta s} \Delta D_s \int_0^{1} c \left(X^{D}_{s^-} + \Delta X_s -\lambda \Delta D_s \right) d\lambda  \right], 
\end{multline*} 
where $\delta>0$ is a fixed discount rate, $f$ is the running cost, and $c$ is the state-dependent marginal cost of control. Precise assumptions on $f$ and $c$ are also stated in the next subsection.

The integral with respect to the singular control reflects the
interpretation of $c$ as a state-dependent marginal intervention cost.
If an intervention of size $\xi$ moves the state from $y$ to $y-\xi$,
its total cost is obtained by accumulating the marginal cost along the
intervention:
\[
\int_{y-\xi}^{y}c(z)\,dz
=
\xi\int_0^1 c(y-\lambda\xi)\,d\lambda.
\]
This definition is invariant under the subdivision of
an intervention into smaller consecutive interventions. An analogous
convention is used for state-dependent marginal rewards in
\cite{JMP}.

The objective of this article is to connect this control problem with an optimal stopping problem, which we call the \emph{associated} OSP. 
To solve the OSP (and hence the control problem), it is natural to introduce the supremum and infimum of the process up to an independent exponential time (see for instance \cite{SURYA} and \cite{M.M(2015)}).
\begin{dfntn}
Given $\delta>0$, consider
\begin{equation*}
S=\sup\{X_t\colon 0\leq t\leq e_{\delta}\},\quad
I=\inf\{X_t\colon 0\leq t\leq e_{\delta}\},
\end{equation*}
where $e_{\delta}$ is an exponential random variable of parameter $\delta$, independent of $X$.
\end{dfntn}
\begin{assumption}\label{A:AssumptionsProcess}
For the uncontrolled process, we assume that there is a constant $\theta>0$ in the domain of $\phi$, such that $\max \left(\phi (\theta), \phi(-\theta)\right)< \delta$.
\end{assumption}
This assumption requires the existence of the exponential moments $\E(e^{\pm \theta X_1})$, which by \cite[Theorem 25.3]{KIS} is equivalent to the integrability condition $$\int_{\vert y \vert > 1} e^{\pm \theta y}\, \Pi(dy) < \infty .$$ Under this condition, the characteristic exponent $\phi$ extends analytically to the strip $\{z\in\xC\colon \mathrm{Re}(z)\in[-\theta,\theta]\}$ (see \cite[Theorem 25.17]{KIS}), so that the values $\phi(\theta)$ and $\phi(-\theta)$ appearing in the assumption are well defined as real numbers. The condition $\max(\phi(\theta),\phi(-\theta))<\delta$ ensures, via the bound $\E(e^{\pm \theta X_t}) = e^{t\phi(\pm\theta)}$, the integrability of the discounted exponential moments needed throughout the paper.
\begin{assumption}\label{A:Assumptions2} 
Regarding the control and running costs we assume the following.
\begin{itemize} 
\item[$(\rm{i})$]\textbf{Control cost.} Let $c\colon\xR\rightarrow\xR$ be a continuously differentiable function in the domain of the infinitesimal generator, 
modeling the marginal cost (or reward, when negative) of applying control.
\item[$(\rm{ii})$]
\textbf{Running cost.} Let $f \colon \xR  \rightarrow \xR$ be a continuously differentiable function modeling the running cost (or reward, when negative) of the controlled process. 
We assume that one of the following conditions holds:
\begin{itemize}
    \item[\rm(a)]
 $\limsup_{x \to -\infty} f(x) \leq M \ \text{for some }M  >0;$
\item[\rm(b)]  $f \vert_{(-\infty,-M]}$ is decreasing for some $M >0$.
\end{itemize}

\item[$(\rm{iii})$]\textbf{Exponential growth bound.}
There is a constant $K>0$ such that 
$$ 
\max \left( \vert f(x)\vert ,\vert f'(x) \vert ,  \vert c(x)  \vert \right) \leq K(1+\cosh(\theta x)) \quad \text{for all } x\in \xR.
$$ 
In particular, this condition implies that the map
\begin{equation}\label{eq:assumptions}
x\ \mapsto \E_x \int_0^{\infty}\vert f(X_s) \vert e^{-\delta s}ds   
\end{equation}
is finite and continuous.
\end{itemize}
\end{assumption}
\begin{rmrk}
The conditions in (ii) of Assumption  \ref{A:Assumptions2} are mild. For instance, they are satisfied by functions that are either concave or convex on some negative half-line $(-\infty, -M]$, which in particular includes any linear combination of a polynomial and exponentials of the form $e^{\pm\alpha x}$ with $0<\alpha\leq\theta$.
\end{rmrk}
We now introduce a function $r$ that will play a key role in the connection between the control problem and the associated OSP. It is defined by

\begin{equation}\label{eq:r}
    r(x)=(\delta- \mathcal{L})c(x).   
\end{equation}
The motivation for this definition is that, under our assumptions, $r$ satisfies $\delta^{-1}\E_x r(X_{e_\delta}) = c(x)$, so that $r$ recovers the control cost $c$ through the resolvent of the killed process. This identity is exploited in Section \ref{S:Verification} to relate $c$ to the OSP value function.

If the uncontrolled process has unbounded variation, the analysis of the associated OSP requires smoothness of its value function, which we ensure through the following additional condition.
\begin{assumption}\label{A:Assumptions3Unboundedvariation}
Suppose the underlying process has unbounded variation. Then we assume that $r$ and $f'$ are continuously differentiable and that
$$
\max \left(\vert r'(x) \vert, \vert f''(x) \vert \right) \leq K(1+ \cosh(\theta x)).
$$
\end{assumption}
\begin{rmrk}
    If the functions $f$ and $c$ and their first and second derivatives are dominated by a polynomial, then Assumption  \ref{A:Assumptions2}$(\rm{iii})$ holds.
\end{rmrk}
\begin{assumption}\label{A:Assumptions4}
We assume that the function
\begin{equation}\label{eq:CONDITION}
Q(x):=\E_x \left(f'- r\right)(I)
\end{equation}
has a root $x^\ast$, is non-positive in $(-\infty,x^{\ast})$,  non-decreasing and strictly positive in $(x^\ast,\infty)$. 
\end{assumption}
\begin{rmrk}
Under Assumption  \ref{A:Assumptions2}, if $f'-r$ is non-decreasing and $\lim_{x \to -\infty}(f'-r)(x)<0<\lim_{x \to \infty}(f'-r)(x)$, then Assumption  \ref{A:Assumptions4} holds; indeed, in this case $Q$ itself is non-decreasing. 
In particular, when $c$ is constant and $f$ is convex, these conditions are satisfied and we recover the setting of \cite{NOBAYAMAZAKI}.
\end{rmrk}
\begin{dfntn}[Admissible controls]\label{D:Admissiblestrategies}
An \emph{admissible control} is a non-negative $\mathbf{F}$-adapted process $D$ such that:
\vskip1mm\par\noindent
{\rm(i)} $D\colon\Omega \times \xR_+ \rightarrow \xR_+$ is right-continuous and non-decreasing almost surely, with $D_{0^-}=0$. The associated controlled process is defined by
\begin{equation*}
X^{D}_t= X_t-D_t , \qquad X_{0^-}=X_0=x,
\end{equation*}
and we write $d_0 := D_0 \geq 0$ for the (possibly nonzero) initial control, i.e., the instantaneous adjustment applied at time $t=0$.
%$$X_t^D:= X_t-D_t. $$
\vskip1mm\par\noindent
{\rm(ii)} $D_t$ has finite expectation for every $t\geq 0$.
\vskip1mm\par\noindent
{\rm(iii)} The following integrability condition holds:
\begin{multline*}
 \E_x \left(\int_{0^+}^{\infty}e^{-\delta s} (\vert f(X_{s^-}^{D})\vert ds + \vert c \left(X^{D}_{s^-}\right)\vert  dD^c_s)  
\right. \\ \left. +    \sum_{0^- < s < \infty}e^{-\delta s}  \Delta D_s \left\vert  \int_0^{1} c \left(X^{D}_{s^-} + \Delta X_s -\lambda \Delta D_s \right)  d\lambda  \right\vert  \right) <\infty.
\end{multline*}
%%%%
\vskip1mm\par\noindent
We denote by $\mathcal{A}$ the set of admissible controls.
\end{dfntn}
An important class of controls, in which the optimal control will be found, is that of \emph{reflecting controls}: those that, with the minimal push, keep the process within a half-line $(-\infty, b]$.
\begin{dfntn}\textbf{Reflecting controls and reflected processes.}
Given $b \in \xR$, define the processes $D^b=\lbrace D_t^b \rbrace_{t \geq 0}$ and $X^b= \lbrace X_t^b \rbrace_{t \geq 0}$ by
$$
D^b_t:=\max \left\lbrace \sup_{0 \leq s \leq t}X_s-b ,\, 0 \right\rbrace , \quad X_t^b := X_t-D_t^b.
$$
We refer to $D^b$ as the \emph{reflecting control at $b$} and to $X^b$ as the \emph{process reflected at $b$}. 
In particular, when $x>b$ the initial value $D_0^b = (x-b)^+$ is positive, corresponding to an instantaneous downward adjustment $d_0^b$ at time $t=0$.
\end{dfntn}
%%%%%%%%%%%%%%%%%%%%%%%%%%%%%%%%%%
%%%%%%%%%%%%%%%%%%%%%%%%%%%%%%%%%%%%%%%%%%%%%%%%
%%%%%%%%%%%%%%%%%%%%%%%%%%%%%%%%%%%%%%%%%%%%%%%%
\subsection{The control problem and main results}
%%%%%%%%%%%%%%%%%%%%%%%%%%%%%%%%%%%%%%%%%%%%%%%%
%%%%%%%%%%%%%%%%%%%%%%%%%%%%%%%%%%%%%%%%%%%%%%%%
%%%%%%%%%%%%%%%%%%%%%%%%%%%%%%%%%%%%%%%%%%%%%%%%
\begin{dfntn}
Given $x\in\xR$ and a control $D \in\mathcal{A}$,
we define the \emph{cost functional} 
\begin{multline*}
 J_{\delta}(x,D):= \E_x \left[
 \int_{0^+}^{\infty}e^{-\delta s} (f(X_{s^-}^{D})ds + c \left(X^{D}_{s^-}\right) dD^c_s)  
 \right. \\ \left. + \sum_{0^- < s < \infty}e^{-\delta s} \Delta D_s \int_0^{1} c \left(X^{D}_{s^-} + \Delta X_s -\lambda \Delta D_s \right) d\lambda  \right],  
\end{multline*}
and the value function
\begin{equation}\label{eq:ValueFunction}
    G_{\delta}(x)= \inf_{D \in \mathcal{A}} J_{\delta}(x,D).
\end{equation}
\end{dfntn}  
We are now in a position to present our main results; their proofs are given in the next section. The first is a Hamilton--Jacobi--Bellman (HJB)-type verification theorem that gives sufficient conditions for the optimality of a reflecting barrier strategy among all admissible controls.
\begin{thrm}\label{T:Verification}
Suppose there exist a constant $x^\ast$ and a function $u\colon\xR\to\xR$ such that:
\begin{itemize}
\item[$(\rm{i})$] $u$ is continuously differentiable if $X$ has bounded variation, and twice continuously differentiable if $X$ has unbounded variation;
\item[$(\rm{ii})$] $\vert u(x) \vert \leq K(1+ \cosh (\theta x))$ for all $x \in \xR$ and some $K>0$;
\item[$(\rm{iii})$] $u$ satisfies the HJB system
\begin{align}\label{eq:TverificationCondition}
&\mathcal{L}u(x)- \delta u(x)+f(x) \geq 0 \quad\text{for all } x \in \xR, 
\text{ with equality for } x<x^\ast, \nonumber \\
& u'(x)\leq c(x)   \quad\text{for all } x \in \xR, \text{ with equality for } x\geq x^\ast.
\end{align}
\end{itemize}
Then, $u$ is the value function of the control problem \eqref{eq:ValueFunction} and
$x^\ast$ is the optimal reflecting barrier:
$$
u(x)=J_{\delta}(x,D^{x^\ast})=G_{\delta}(x).
$$
\end{thrm}
Our second main result characterizes the verification function $u$ and the optimal threshold $x^\ast$ of Theorem \ref{T:Verification} through the solution of an associated OSP.
\begin{thrm}\label{T:DISCOUNTEDPROBLEMSOLUTION}
Consider the optimal stopping problem
\begin{equation}\label{eq:OSP}
    v(x)= \inf_{\tau \in \mathcal{R}}\E_x \left( \left(\int_0^{\tau}f'(X_s)e^{- \delta s}ds +c(X_\tau)  e^{-\delta \tau} \right)\mathbf{1}_{\tau < \infty}
+ \left(\int_0^\infty f'(X_s)e^{-\delta s}ds \right)\ \mathbf{1}_{\tau=\infty}    \right)
\end{equation}
where $\mathcal{R}$ is the set of $\mathbf{F}$-stopping times, and let $x^\ast$ be the root introduced in Assumption \ref{A:Assumptions4}. 
Then the stopping time
\begin{equation}\label{eq:taustar}
    \tau^\ast=\inf\{t\geq 0\colon X_t\geq x^\ast\}
\end{equation}
is optimal, that is
\begin{equation}\label{eq:OSPstar}
v(x)=\E_x \left( \left(\int_0^{\tau^\ast}f'(X_s) e^{- \delta s}ds +c (X_{\tau^\ast}) e^{-\delta \tau^\ast}\right)\mathbf{1}_{\tau^\ast< \infty}+ \left(  \int_0^\infty f'(X_s)e^{-\delta s}ds\right)\mathbf{1}_{\tau^\ast= \infty} 
\right).\end{equation}
Furthermore, a suitably chosen primitive of $v$ satisfies the conditions
of Theorem \ref{T:Verification}; 
in other words, the function $u$ of Theorem \ref{T:Verification} satisfies $u'=v$.
\end{thrm}
\section{Proofs of the main results}\label{S:Verification}
\subsection{Verification Theorem}
In order to prove Theorem \ref{T:Verification}, we first establish some preliminary results showing that the problem is nontrivial, i.e., that $G_{\delta}(x)<\infty$ for all $x \in \xR$.
\begin{prpstn}\label{P:INFINITESIMALGENERATOR}
\begin{itemize}
    \item[$(\rm{i})$]
Suppose $X$ has unbounded variation, and let $u \in C^2(\xR)$ satisfy
 $$\max \left( \vert u(x) \vert , \vert u'(x) \vert , \vert u''(x)\vert \right) \leq K(1+\cosh (\theta x)).$$
Then $u$ belongs to the domain of the infinitesimal generator, $\mathcal{L} u$ is continuous, and
    $$ 
\mathcal{L} u(x)
  = \mu u'(x)
  + \frac{\sigma^2}{2} u''(x)
  + \int_{\xR}
     \big(
       u(x + y) - u(x) - u'(x) y
     \big)\, \Pi(dy).
 $$
\item[$(\rm{ii})$]
Suppose $X$ has bounded variation, and let $u \in C^1(\xR)$ satisfy
$$ 
\max( \vert u(x) \vert , \vert u'(x) \vert )  \leq K(1+\cosh (\theta x)).
$$
Then $u$ belongs to the domain of the infinitesimal generator, $\mathcal{L} u$ is continuous, and
%$$ 
%\mathcal{L} u(x)
%  = \mu u'(x)
%  + \int_{\xR}
%     \big(
%      u(x + y) - u(x)
 %    \big)\, \Pi(dy).
 %$$
$$ 
\mathcal{L}u(x)
=
\mu u'(x)
+
\int_{\xR}
\left(
u(x+y)-u(x)-u'(x)y
\right)\Pi(dy).
$$

\end{itemize}
\end{prpstn}

\begin{proof}
First, observe that the operators defined on the right-hand side of the equalities in $(\rm{i})$ and $(\rm{ii})$ are well-defined and continuous. Indeed, this follows from the integrability condition
$$\int_{\vert y \vert\geq 1} \cosh( \theta y) \  \Pi(dy)< \infty.  $$
Let us prove $(\rm{i})$. Define 
    $$ 
\mathcal{L}^\ast u(x)
  = \mu u'(x)
  + \frac{\sigma^2}{2} u''(x)
  + \int_{\xR}
     \big(
       u(x + y) - u(x) - u'(x) y
     \big)\, \Pi(dy).
 $$
By \cite[Theorem 31.5]{KIS}, the classical Dynkin formula yields
\begin{equation*}
\E_x u(X_{t \wedge T_n})-u(x) =  \E_x \int_0^{t \wedge T_n} \mathcal{L}^\ast u(X_s)\, ds, 
\end{equation*}
with $T_n= \inf \lbrace t\colon\vert X_t \vert \geq n \rbrace$. Since both $\vert u(x) \vert$ and $\vert \mathcal{L}^\ast u(x) \vert$ are dominated by $K(1+\cosh(\theta x))$, dominated convergence (as $n \to \infty$) gives
\begin{equation*}
\E_x u(X_t) - u(x) = \E_x \int_0^{t} \mathcal{L}^\ast u(X_s)\, ds.
\end{equation*}
Dividing by $t$ and letting $t \to 0^+$, the right-hand side converges to $\mathcal{L}^\ast u(x)$ by continuity of $\mathcal{L}^\ast u$ and \cite[Ch.~I, 1.2.C]{D}, which proves $(\rm{i})$. 
The case of bounded variation follows the same argument, with \cite[Theorem 31.5]{KIS} replaced by \cite[Chapter II, Theorem 31]{ProtterPE}.
\end{proof}
\begin{prpstn}\label{P:finitude}
For every $b\in\xR$, the reflecting control $D^b$ is admissible, and $J_\delta(x, D^b) < \infty$ for all $x \in \xR$. Consequently, $G_{\delta}(x)< \infty$ for all $x \in \xR$. 
\end{prpstn}  
\begin{proof}
The last statement is a corollary of the previous ones. 
We will show that
\begin{align}
   &\E_x\left(\int_{0^+}^{\infty}e^{-\delta s}\left(\vert f(X_{s^-}^{b})\vert ds +\vert c (X^{b}_{s^-})\vert d(D^b)^c_s\right)  \right)<\infty, \label{eq:finitude1} \\   
   & \E_x  \sum_{0 < s < \infty}e^{-\delta s}   \Delta D^b_s \left\vert \int_0^{1} c \left(X^{b}_{s^-} + \Delta X_s -\lambda \Delta D^b_s \right) d\lambda  \right\vert <\infty.\label{eq:finitude2}
\end{align}   
We first study \eqref{eq:finitude1}. The term 
$$ 
\E_x   \int_{0^+}^{\infty}e^{-\delta s}   \vert c \left(X^{b}_{s^-}\right)\vert d(D^b)^c_s
=\E_x   \int_{0^+}^{\infty}e^{-\delta s}   \vert c \left(b\right)\vert d(D^b)^c_s
$$
is finite by \cite[A.2(\rm{ii})]{NOBAYAMAZAKI}, since the continuous part of $D^b$ increases only at the boundary $b$.
For the remaining term, we cannot apply \cite[A.2(\rm{i})]{NOBAYAMAZAKI} directly, since that result assumes $\vert f \vert$ is dominated by a polynomial. However, that hypothesis is only used to establish the inequality.
\begin{equation}\label{eq:finitudeyamazaki}
 \int_{(-\infty,-1)}\sup_{z \in [-1,0]} \left(\E_{z+y}\int_0^\infty e^{-\delta t}\vert f(X_t)\vert dt \right) \Pi(dy)   <\infty,
\end{equation}
which we now verify under our weaker exponential growth hypothesis.
For $z \in [-1,0]$, by Assumption  \ref{A:Assumptions2}(\rm{iii}) together with Assumption  \ref{A:AssumptionsProcess},
\begin{multline*}
 \E_{z+y}\int_0^{\infty}e^{-\delta t}\vert f(X_t)\vert dt \leq K\E \int_0^\infty \frac{e^{-\delta t}}{2}(e^{\theta (X_t+z+y)}+e^{-\theta (X_t+z+y)}+2) dt \\
 \leq K e^{\theta \vert  y \vert} \E \int_0^{\infty} \frac{e^{-\delta t}}{2}(e^{\theta X_t}+e^{-\theta( X_t-1)}+2) dt+K\delta^{-1}=\hat{K} e^{\theta \vert y \vert}+K\delta^{-1},
\end{multline*}
for some constant $\hat{K}$. Therefore the left-hand side of \eqref{eq:finitudeyamazaki} is bounded above by
$$\int_{(-\infty,-1)}e^{\theta \vert y  \vert} ( \hat{K}+K\delta^{-1})  \Pi(dy) <\infty,$$
where finiteness follows from Assumption  \ref{A:AssumptionsProcess}. This establishes \eqref{eq:finitude1}.

Now we study \eqref{eq:finitude2}. We split the sum according to the size of $\Delta X_s$. For the small jumps,
\begin{multline*}
 \E_x  \sum_{0<s<\infty}e^{-\delta s}\Delta D_s^b \indicator{0 \leq \Delta X_s \leq 1}  \int_0^1 \vert c\vert (X^b_{s^-}+\Delta X_s - \lambda \Delta D^b_s) d\lambda 
\\ \leq  \E_x  \sum_{0<s<\infty}e^{-\delta s}\Delta D_s^b  \sup_{z \in [0,1]}\vert c(z+b) \vert ,
\end{multline*}
which is finite by \cite[A.2(\rm{i})]{NOBAYAMAZAKI}. 
It remains to bound
\begin{equation}\label{eq:finitude3}
    \E_x  \sum_{0<s<\infty}e^{-\delta s}\Delta D_s^b  \indicator{\Delta X_s > 1}  \int_0^1 \vert c \vert (X^b_{s^-}+\Delta X_s - \lambda \Delta D^b_s) d\lambda.
\end{equation}
Using $\vert c(x)\vert \leq K(1 + \cosh(\theta x))$ and noting that $X^b_{s^-} + \Delta X_s - \lambda \Delta D^b_s \geq b$ for $\lambda \in [0,1]$ (so that $\cosh$ is dominated by an exponential up to a constant depending on $b, \theta$), there exists $\overline{K}>0$ such that \eqref{eq:finitude3} is bounded above by
\begin{multline*}
     \E_x  \sum_{0<s<\infty}e^{-\delta s}\Delta D_s^b  \indicator{\Delta X_s > 1}  \int_0^1 \overline{K}\left(1+ \exp( \theta ( X^b_{s^-}+\Delta X_s - \lambda \Delta D^b_s) )\right)d\lambda \\
     = \E_x  \sum_{0<s<\infty}e^{-\delta s}  \indicator{\Delta X_s > 1}  \overline{K}\left( \Delta D_s^b +\theta^{-1}( e^{ \theta ( X^b_{s^-}+\Delta X_s )}-e^{\theta X_s^b})\right)\\ 
     \leq \E_x  \sum_{0<s<\infty}e^{-\delta s} \indicator{\Delta X_s > 1}  \overline{K}\left( \Delta D_s^b +\theta^{-1} e^{ \theta ( b+\Delta X_s )}\right).
\end{multline*} 
Since $\Delta D_s^b \leq \Delta X_s$ on $\{\Delta X_s \geq 1\}$ 
(a jump of $D^b$ cannot exceed the corresponding jump of $X$),  
\eqref{eq:finitude3} is bounded above by
\begin{equation*}
 \E_x  \sum_{0<s<\infty}e^{-\delta s} \indicator{\Delta X_s > 1}  \overline{K}\left(\Delta X_s + \theta^{-1}e^{ \theta ( b+\Delta X_s )}\right).
    \end{equation*}
By the compensation formula for the jump measure of $X$, this equals
\begin{equation*}
       \int_0^{\infty}  \int_{(1,\infty)}e^{-\delta s}  \overline{K}  (y+\theta^{-1}e^{\theta (b+y)}  )  \, \Pi(dy)\, ds =  \frac{\overline{K}e^{\theta b}}{\delta}\int_{(1,\infty)} (e^{-\theta b}y+\theta^{-1}e^{\theta y})\,\Pi (dy).
\end{equation*}
The right-hand side is finite by Assumption  \ref{A:AssumptionsProcess}, since $\int_{(1,\infty)} e^{\theta y} \Pi(dy) < \infty$ whenever $\theta$ lies in the domain of $\phi$, and $\int_{(1,\infty)} y\,\Pi(dy) < \infty$ by the finite-mean assumption. 
This shows \eqref{eq:finitude3} is finite, completing the proof.
\end{proof}
In order to establish Theorem \ref{T:Verification}, we first prove the following two lemmas.
\begin{lmm}\label{L:inequality}
    Under the hypotheses of Theorem \ref{T:Verification}, there exists an increasing localizing sequence of stopping times $\lbrace T_n \rbrace_{n \in \xN}$ such that
$$u(x) \leq J_{\delta}(x,D)+\liminf_{n \to \infty}\limsup_{t \to \infty}\E_x  e^{-\delta (t\wedge T_n)}u(X_{t\wedge T_n}^D)   $$
for every $D \in \mathcal{A}$. In particular, 
\begin{equation}\label{eq:bound0}
    u(x) \leq J_{\delta}(x,0)=  \delta^{-1}\E_x f(X_{e_{\delta}}), 
    \end{equation}
where $0$ denotes the null control $D \equiv 0$.
\end{lmm}    

\begin{proof}
First, assume $X$ has unbounded variation and $u \in C^2(\xR)$.

Define the local martingales $\{M^{(i)}_t\}_{t\geq 0}$ ($i=1,2$) by
\begin{align*}
    M^{(1)}_t&=\sigma W_t+\int_{(0,t]\times \xR}y\,\tilde N(ds,dy),\\
    M^{(2)}_t&=\int_{(0,t]\times \xR}e^{-\delta s}\left(
    u(X^{D}_{s^-}+y)-u(X^{D}_{s^-})-yu'(X^{D}_{s^-})
    \right)\tilde N(ds,dy).
\end{align*}
Let $T'_n$ be a common localizing sequence for $M^{(1)}$ and $M^{(2)}$, and define
$$
S_n= \inf\lbrace t\geq 0\colon \max( \vert X_t \vert , \vert X^D_{t^-}+ \Delta X_t   \vert , \vert X_t^D  \vert ) \geq n \rbrace.
$$
Setting $T_n = T'_n \wedge S_n$ yields a localizing sequence such that $\max(|X_t|, |X^D_{t^-} + \Delta X_t|, |X^D_t|) \leq n$ for all $t < T_n$.
Applying It\^o's formula to $e^{-\delta t} u(X^D_t)$ on $[0, t \wedge T_n]$, we obtain
\begin{align}
    e^{-\delta (t \wedge T_n)}&u(X^{D}_{t \wedge T_n })-u(x-d_0) \notag \\ &=\int_{0^+}^{t \wedge T_n}e^{-\delta s}u'(X^{D}_{s^-})dX^{D}_s+\frac{\sigma^2}{2}\int_{0^+}^{t \wedge T_n}e^{-\delta s}u''(X^{D}_{s^-})\,ds+M^{(2)}_{t \wedge T_n}\notag\\
    &\quad+\int_{(0,t \wedge T_n]\times \xR}e^{-\delta s}\left(
    u(X^{D}_{s^-}+y)-u(X^{D}_{s^-})-yu'(X^{D}_{s^-})
    \right)ds\Pi(dy)\notag\\
    &\quad -\int_{0^+}^{t \wedge T_n}\delta e^{-\delta s} u(X_{s^-}^D)ds \nonumber \\ & +\sum_{0<s \leq t \wedge T_n }\left(u(X^{D}_{s})-u(X^{D}_{s^-}+\Delta X_s)+u'(X_{s^-}^{D})\Delta D_s\right) e^{-\delta s}  \nonumber \\
    &=\int_{0^+}^{t \wedge T_n}e^{-\delta s}\mathcal{L}u(X^{D}_{s^-})\,ds+\int_{0^+}^{t \wedge T_n}e^{-\delta s}u'(X^{D}_{s^-})\,dM^{(1)}_s+M^{(2)}_{t \wedge T_n}\notag\\
    &\quad-\int_{0^+}^{t \wedge T_n}e^{-\delta s}u'(X^{D}_{s^-})\,dD^c_s +\sum_{0<s \leq t \wedge T_n}e^{-\delta s} \left( u(X_s^{D})-u(X_{s^-}^{D}+\Delta X_s) \right) \nonumber \\
    & -\int_{0^+}^{t \wedge T_n}\delta e^{-\delta s} u(X_{s^-}^D)\,ds.\label{useful}
\end{align}
Notice that
$$u(X_s^{D})-u(X_{s^-}^{D}+\Delta X_s) =-\Delta D_s \int_0^{1} u'(X_{s^-}^D + \Delta X_s -\lambda \Delta D_s) d\lambda,$$
by the fundamental theorem of calculus applied along the segment from $X^D_{s^-}+\Delta X_s$ to $X^D_s = X^D_{s^-}+\Delta X_s - \Delta D_s$.
Therefore, by \eqref{useful} and the HJB inequalities 
$\mathcal{L}u - \delta u + f \geq 0$ and $u' \leq c$, 
we obtain 
\begin{align*}
    e^{-\delta(t \wedge T_n)}u(X^{D}_{t \wedge T_n})&-u(x-d_0) \\ 
    &\geq \int_{0^+}^{t \wedge T_n}e^{-\delta s}\left(\delta u-f \right)(X^{D}_{s^-})\,ds\\
    & +M ^{(3)}_{t \wedge T_n}-\int_{0^+}^{t \wedge T_n}e^{-\delta s}c  (X^{D}_{s^-})  \,dD^c_s   \\ 
   & -\sum_{0 <s \leq{t\wedge T_n} }  e^{-\delta s} \Delta D_s \int_0^{1} c (X_{s^-}^D + \Delta X_s -\lambda \Delta D_s) d\lambda \\
    & -\int_{0^+}^{t \wedge T_n}\delta e^{-\delta s}u(X_{s^-}^D)\,ds,
\end{align*}
where 
$M^{(3)}_t:=\int_{0^+}^{t}e^{-\delta s}u'(X^{D}_{s^-})\,dM^{(1)}_{s}+M^{(2)}_{t}$ 
is a local martingale with localizing sequence $\lbrace T_n \rbrace_{n \in \xN}$. 
Taking expectations (the local martingale term vanishes since $M^{(3)}_{t \wedge T_n}$ is a true martingale) and rearranging, 
\begin{multline}\label{eq:inequalityverification2}
  \E_x \left(   \int_{0^+}^{t \wedge T_n}e^{-\delta s} f (X^{D}_{s^-})\,ds +\int_{0^+}^{t \wedge T_n}e^{-\delta s} c(X^{D}_{s^-}) \,dD^c_s   \right.  \\
      \left. +\sum_{0 <s \leq{t\wedge T_n} } e^{-\delta s}  \Delta D_s \int_0^{1} c (  X_{s^-}^D + \Delta X_s -\lambda  \Delta D_s)\,  d\lambda \right) \\ 
 + \E_x  e^{-\delta (t \wedge T_n)}  u(X^{D}_{t\wedge T_n })   \geq  u(x-d_0).
\end{multline}
Observe that
$$
u(x - d_0) = u(x) - \int_0^{d_0} u'(x - y)\, dy.
$$
Since $u'(x) \leq c(x)$, 
$$
u(x - d_0) \geq u(x) - \int_0^{d_0} c(x - y)\, dy.
$$
By the change of variables $y = \lambda d_0$,
$$
u(x - d_0) \geq u(x) - d_0 \int_0^1 c(x - \lambda d_0)\, d\lambda.
$$
Substituting this into \eqref{eq:inequalityverification2}, taking $\limsup_{t \to \infty}$ and then $\liminf_{n \to \infty}$ on both sides yields the first claim for the case of unbounded variation.

The case of bounded variation follows similarly, applying the It\^o--Meyer formula for c\`adl\`ag processes of finite variation (see \cite[Chapter II, Theorem 31]{ProtterPE}).
Finally, to prove \eqref{eq:bound0}, we use Assumption \ref{A:Assumptions2}(iii) and Assumption  \ref{A:AssumptionsProcess}:
$$
\limsup_{t \to \infty}\E_x\, e^{-\delta t} \vert u (X_t)\vert \leq \limsup_{t\to \infty}  e^{-\delta t} \E_x K(1+\cosh (\theta X_t)) =0,
$$
where the last equality uses $\E_x \cosh(\theta X_t) \leq \cosh(\theta x) e^{\max(\phi(\theta), \phi(-\theta))\, t}$ together with $\max(\phi(\theta), \phi(-\theta)) < \delta$.
\end{proof}
\begin{lmm}\label{Lemma:hardbound}
    Under the hypotheses of Theorem \ref{T:Verification}, for every $x \in \xR$ and every $D \in \mathcal{A}$, the localizing sequence $\{T_n\}$ from Lemma \ref{L:inequality} satisfies
\begin{equation*}
\limsup_{n\to \infty}\limsup_{t \to \infty}\E_x e^{-\delta (t\wedge T_n)} u(X_{t\wedge T_n}^D)\leq 0.
\end{equation*}
\end{lmm}
\begin{proof}
By Lemma \ref{L:inequality} applied to the null control,
\begin{multline*}
        \E_x e^{-\delta (t \wedge T_n)} u(X_{t\wedge T_n }^D) \leq \E_x e^{-\delta (t\wedge T_n)}J_\delta(X^D_{t \wedge T_n},0)   \\    =\E_x e^{-\delta (t\wedge T_n)} \int_0^{\infty}e^{-\delta s}f(X_{s+t\wedge T_n}-X_{t\wedge T_n}+X_{t\wedge T_n}^D)ds, 
\end{multline*}
where the last equality follows from the strong Markov property: the increment $X_{s+t\wedge T_n}-X_{t\wedge T_n}$ is independent of $\mathcal{F}_{t\wedge T_n}$ and has the same distribution as $X_s$.

We now distinguish the two cases of Assumption \ref{A:Assumptions2}(ii). 
If (ii)(a) holds, then since $f$ is continuous and $\limsup_{x\to-\infty} f(x) \leq M$, $f$ is bounded above on $(-\infty,-M]$  by some positive constant $M'$. 
Hence
\begin{multline}\label{eq:firstindicatrix}
\E_x e^{-\delta (t\wedge T_n)}\int_0^{\infty}e^{-\delta s} f(X_{s+t\wedge T_n}-X_{t\wedge T_n}+X_{t
\wedge T_n}^D) \indicator{X_{s+t\wedge T_n}-X_{t\wedge T_n}+X_{t\wedge T_n}^D \leq -M} \ ds \\ \leq M'\delta^{-1}\E_x e^{-\delta (t\wedge T_n)}.    
\end{multline}
If instead (ii)(b) holds, i.e., there exists $M>0$ such that $f$ is decreasing on $(-\infty,-M]$, 
we use that
$$
X_{s+t\wedge T_n}-X_{t \wedge T_n}+X^D_{t \wedge T_n} \geq X^D_{s+t \wedge T_n }
$$
(since $D$ is non-decreasing) to obtain, on $\{X_{s+t\wedge T_n}-X_{t\wedge T_n}+X^D_{t\wedge T_n} \leq -M\}$, that $X^D_{s+t\wedge T_n} \leq -M$ and hence $f(X_{s+t\wedge T_n}-X_{t\wedge T_n}+X^D_{t\wedge T_n}) \leq f(X^D_{s+t\wedge T_n})$ by monotonicity. 
Therefore, using $f(x)-f(-M)\geq 0$ for all $x\leq -M$,
\begin{multline}\label{eq:firstindicatrixother}
\E_x e^{-\delta (t\wedge T_n)} \int_0^{\infty}e^{-\delta s} (f(X_{s+t\wedge T_n}-X_{t\wedge T_n}+X_{t
\wedge T_n}^D)-f(-M)) \indicator{X_{s+t\wedge T_n}-X_{t\wedge T_n}+X_{t\wedge T_n}^D \leq -M} \ ds \\ \leq \E_x \int_{t \wedge T_n}^{\infty} e^{-\delta s}(f(X^D_{s})-f(-M)) \indicator{X^D_{s} \leq -M}\, ds.    
 \end{multline}
Moreover, for every fixed $n$, as $t\to\infty$ we have
$t\wedge T_n\to T_n$. Therefore,
\[
\E_x\int_{t\wedge T_n}^{\infty}e^{-\delta u}\,du
\longrightarrow
\frac{1}{\delta}\E_x e^{-\delta T_n}.
\]
Since $T_n\to\infty$ almost surely as $n\to\infty$, it follows that
\begin{equation}\label{eq:minimumdontbother}
\lim_{n\to\infty}\lim_{t\to\infty}
|f(-M)|\,
\E_x\int_{t\wedge T_n}^{\infty}e^{-\delta u}\,du
=0.
\end{equation}
For the complementary event, in both cases,
\begin{multline}\label{eq:secondindicatrix}
 \E_x  e^{-\delta (t\wedge T_n)}   \int_0^{\infty}  e^{-\delta s}  f(X_{s+t\wedge T_n}-X_{t\wedge T_n}+X_{t\wedge T_n}^D)  \indicator{X_{t\wedge T_n+s}-X_{t \wedge T_n}+X_{t\wedge T_n}^D > -M}  ds  \\
\leq  \E_x  e^{-\delta (t\wedge T_n)}     \int_0^{\infty}e^{-\delta s}  K(1+\cosh(\theta( X_{s+t\wedge T_n}-X_{t\wedge T_n} +X_{t\wedge T_n}^D))) \\
\times \indicator{X_{t\wedge T_n+s}-X_{t \wedge T_n}+X_{t\wedge T_n}^D  > -M}\,  ds  .    
\end{multline}

On $\{X_{s+t\wedge T_n}-X_{t\wedge T_n}+X^D_{t\wedge T_n} > -M\}$, $e^{-\theta(\cdots)} < e^{\theta M}$, so $1+\cosh(\theta(\cdots))
\leq
\frac{1}{2}
\left(2+e^{\theta M}+e^{\theta(\cdots)}\right)$, the right-hand side of \eqref{eq:secondindicatrix} is bounded above by
\begin{equation}\label{eq:secondindicatrix2}
 \frac{K}{2}\E_x \int_0^{\infty}e^{-\delta (s+t\wedge T_n)} (2+e^{\theta M}+e^{\theta X_{s+t\wedge T_n}}) ds  = \frac{K}{2}\E_x \int_{t \wedge T_n}^{\infty}e^{-\delta u}(2+e^{\theta M}+ e^{\theta X_{u}}) du .  
\end{equation}
We now combine the bounds. 
As $t \to \infty$ for fixed $n$, $t \wedge T_n \to T_n$, and as $n \to \infty$, $T_n \to \infty$ almost surely. In case (ii)(a), combining \eqref{eq:firstindicatrix} and \eqref{eq:secondindicatrix2}, both right-hand sides tend to zero by dominated convergence (using Assumption \ref{A:AssumptionsProcess} for the second). In case (ii)(b), combining \eqref{eq:firstindicatrixother}, \eqref{eq:minimumdontbother} and \eqref{eq:secondindicatrix2}, the same conclusion follows, where the first integral tends to zero by the integrability condition in Definition \ref{D:Admissiblestrategies}(iii).
\end{proof}
With the two previous lemmas, we have enough tools to prove the verification Theorem \ref{T:Verification} 
\begin{proof}[Proof of Theorem \ref{T:Verification}]
Let $D \in \mathcal{A}$. Combining Lemma \ref{L:inequality} with Lemma \ref{Lemma:hardbound}, we obtain $u(x) \leq J_\delta(x, D)$. Since $D$ was arbitrary, $u(x) \leq G_\delta(x)$.

It remains to show that the reflecting control $D^{x^\ast}$ achieves equality, that is, $u(x) = J_\delta(x, D^{x^\ast})$. For $D = D^{x^\ast}$, the controlled process satisfies $X^{x^\ast}_t \leq x^\ast$ for all $t \geq 0$, with the continuous part of $D^{x^\ast}$ increasing only at the boundary $x^\ast$ and jumps of $D^{x^\ast}$ occurring only when $X$ jumps above $x^\ast$. Consequently, the HJB inequalities \eqref{eq:TverificationCondition} hold with equality at every relevant point of evaluation, and the inequality \eqref{eq:inequalityverification2} in the proof of Lemma \ref{L:inequality} becomes an equality:
\begin{multline*}
  \E_x \left(   \int_{0^+}^{t}e^{-\delta s} f (X^{x^\ast}_{s^-})ds +\int_{0^+}^{t}e^{-\delta s} c(X^{x^\ast}_{s^-}) d(D^{x^\ast})^c_s   \right.  \\
      \left. +\sum_{0 <s \leq t} e^{-\delta s}  \Delta D^{x^\ast}_s \int_0^{1} c (  X_{s^-}^{x^\ast} + \Delta X_s -\lambda  \Delta D^{x^\ast}_s)  d\lambda \right) \\ 
    =  u(x-d_0)-\E_x \, e^{-\delta t}  u(X^{x^\ast}_{t}).
\end{multline*}
To conclude, it suffices to show that $\E_x e^{-\delta t} u(X^{x^\ast}_t) \to 0$ along a suitable sequence $t_n \to \infty$. By the same argument used in the proof of Proposition \ref{P:finitude} for $\E_x \int_{0^+}^\infty e^{-\delta s} \vert f \vert (X^b_s)ds$, 
one obtains
$$\E_x \int_{0^+}^\infty e^{-\delta s} \cosh (\theta X_s^{x^\ast})\,ds< \infty.$$
Since the integrand is non-negative and integrable, 
$$
\liminf_{t \to \infty} e^{-\delta t} \E_x \cosh(\theta X^{x^\ast}_t) = 0,
$$
so there exists a sequence $t_n \to \infty$ with
$$
\lim_{n \to \infty}\E_x  e^{-\delta t_n }  \cosh(\theta X^{x^\ast}_{t_n})=0.
$$
Combined with the growth bound 
$\vert u(x) \vert \leq K(1 + \cosh(\theta x))$, 
this gives 
$$
\lim_{n \to \infty} \E_x e^{-\delta t_n} u(X^{x^\ast}_{t_n}) = 0.
$$
Letting $t = t_n \to \infty$ in the equality above yields $u(x) = J_\delta(x, D^{x^\ast})$, and hence $u(x) = G_\delta(x) = J_\delta(x, D^{x^\ast})$, completing the proof.
\end{proof}
%%%%%%%%%%%%%%%%%%%%%%%%%%%%%%%%%%%%%%%%%%%%%%%%%%
%%%%%%%%%%%%%%%%%%%%%%%%%%%%%%%%%%%%%%%%%%%%%%%%%%
%%%%%%%%%%%%%%%%%%%%%%%%%%%%%%%%%%%%%%%%%%%
\subsection{The associated OSP and proof of Theorem \ref{T:DISCOUNTEDPROBLEMSOLUTION}}
We study the properties of the associated optimal stopping problem (OSP) in order to prove Theorem \ref{T:DISCOUNTEDPROBLEMSOLUTION}, the main result of the article. Define
$$
H(x):=\E_x \int_0^{\infty} f'(X_s)e^{-\delta s}ds.
$$
\begin{lmm}\label{L:optimalstopping}
The largest root $x^{\ast}$ of the function $Q$ defined in \eqref{eq:CONDITION}
is the optimal threshold for the optimal stopping problem \eqref{eq:OSP}; that is, 
the stopping time defined in \eqref{eq:taustar} satisfies \eqref{eq:OSPstar}.
\end{lmm}
\begin{proof}
To match the notation of \cite{SURYA} and ease the exposition, observe that
\begin{multline}\label{eq:OSP2}
v(x)= -\sup_{\tau \in \mathcal{R}}\E_x\left( \left(\int_0^{\tau}-f'(X_s)e^{- \delta s}ds -c( X_{\tau}) e^{-\delta \tau} \right)\mathbf{1}_{\tau < \infty} \right. \\ \left.
- \left(\int_0^\infty f'(X_s)e^{-\delta s}ds \right)\ \mathbf{1}_{\tau=\infty} \right)   ,
\end{multline}
which can be rewritten as 
\begin{equation}\label{eq:OSP3}
v(x)=H(x) -\sup_{\tau \in \mathcal{R}}\E_x (H-c)(X_\tau )e^{- \delta \tau}\mathbf{1}_{\tau <\infty}.
\end{equation}
On the other hand, by the definition of $r$ in \eqref{eq:r},
$$
c(x)=\delta^{-1}\E_x r(X_{e_{\delta}}). 
$$
Therefore, by the Wiener--Hopf factorization and the definition of $Q$ in Assumption \ref{A:Assumptions4},
$$
(H-c)(x)=\delta^{-1} \E_x Q({S}). 
$$

By Assumption \ref{A:Assumptions2}(\rm{iii}),  for every $x \in \xR$,
\[
\begin{aligned}
|H(x)|
&\leq
\E\int_0^\infty
   \left|f'(x+X_s)\right|e^{-\delta s}\,ds\\
&\leq
K\int_0^\infty e^{-\delta s}
   \left(1+\E\left[\cosh\bigl(\theta(x+X_s)\bigr)\right]\right)ds.
\end{aligned}
\]
Using
\[
\cosh(a+b)\leq 2\cosh(a)\cosh(b)
\]
and
\[
\E\left(\cosh(\theta X_s)\right)
=
\frac{1}{2}
\left(e^{s\phi(\theta)}+e^{s\phi(-\theta)}\right),
\]
we obtain
\[
\begin{aligned}
|H(x)|
&\leq
\frac{K}{\delta}
+
K\cosh(\theta x)
\left(
\frac{1}{\delta-\phi(\theta)}
+
\frac{1}{\delta-\phi(-\theta)}
\right).
\end{aligned}
\]
The two denominators are strictly positive by
Assumption \ref{A:AssumptionsProcess}. Combining this estimate with
the growth bound for $c$ in
Assumption \ref{A:Assumptions2}(\rm{iii}), we conclude that there is
a constant $\widehat K>0$ such that
\begin{equation}\label{eq:growth-H-c}
|H(x)|+|c(x)|
\leq
\widehat K\bigl(1+\cosh(\theta x)\bigr),
\qquad x\in\mathbb R.
\end{equation}
Also, for each
$\lambda\in\{\theta,-\theta\}$, the process
\[
M_t^\lambda
:=
\exp\left\{
\lambda(X_t-x)-t\phi(\lambda)
\right\},
\qquad t\geq0,
\]
is a non-negative martingale under $\mathbb P_x$ with expectation 1. Therefore,
$M_t^\lambda$ converges almost surely, as $t$ tends to infinity, to a finite random variable.
Moreover,
\[
e^{-\delta t}e^{\lambda X_t}
=
e^{\lambda x}
e^{-(\delta-\phi(\lambda))t}M_t^\lambda.
\]
Since $\delta-\phi(\lambda)>0$, it follows that
\[
\lim_{t\to\infty}e^{-\delta t}e^{\lambda X_t}=0
\qquad \mathbb P_x\text{-a.s.},
\qquad \lambda\in\{\theta,-\theta\}.
\]
Consequently,
\[
\lim_{t\to\infty}e^{-\delta t}
\bigl(1+\cosh(\theta X_t)\bigr)=0
\qquad \mathbb P_x\text{-a.s.}
\]
Together with \eqref{eq:growth-H-c}, this yields
\begin{equation}\label{eq:terminalconditionOSP}
\mathbb P_x\left(
\lim_{t\to\infty}
e^{-\delta t}\delta^{-1}(H-c)(X_t)=0
\right)=1.
\end{equation}

Therefore, by Assumption \ref{A:Assumptions2}(iii) and \cite[Theorem 4.2]{SURYA}, the stopping time $\tau^\ast$ is optimal for \eqref{eq:OSP2}, and equivalently for \eqref{eq:OSP}, completing the proof.

\end{proof}
To analyze the continuity and---in the case of unbounded variation---the smoothness of the value function \eqref{eq:OSP3}, we first record some regularity properties of the averaging function $Q$ in the following proposition.
\begin{prpstn}\label{P:PropertiesofQ}
If $X$ has bounded variation, the function $Q$ in \eqref{eq:CONDITION} is continuous. If $X$ has unbounded variation, $Q \in C^1(\xR)$ and
$$
Q'(x)=\E_x \bigl( (f''-r') (I) \bigr).
$$
\end{prpstn}
The proof follows by dominated convergence under Assumptions \ref{A:Assumptions2} and \ref{A:Assumptions3Unboundedvariation}, and is omitted.

With the regularity of $Q$ established, we now turn to the value function $v$ itself.
\begin{prpstn}\label{P:FIT}
The value function $v$ is continuous and there is a constant $K'$ such that 
$$
\vert v(x)\vert \leq K'(1+\cosh(\theta x))\quad \text{for all } x \in \xR.
$$
\end{prpstn}
\begin{proof}
Let $\hat{v}(x):= H(x)-v(x)$.
By Assumption \ref{A:Assumptions2}(iii) and dominated convergence, $H$ is continuous; hence it suffices to prove that $\hat{v}$ is continuous. 
Using \cite[Theorem 4.2]{SURYA} and the same notation as in the previous lemma,
$$ 
\hat{v}(x)= \delta^{-1}\E_x\left[Q({S}) \indicator{{S}\geq x^\ast}\right]= \delta^{-1}\E\left[Q(S+x)\indicator{S+x \geq x^\ast}\right]. 
$$
Since $Q \leq 0$ on $(-\infty, x^\ast]$, we have
\begin{equation*}
\limsup_{h \to 0} \bigl(\hat{v}(x+h)-\hat{v} (x)\bigr)  
\leq \limsup_{h \to 0} \delta^{-1} \E \bigl[\bigl( Q({S}+x+h)- Q({S}+x)  \bigr)  \indicator{{S}+x+h \geq x^\ast}\bigr] .
\end{equation*}
By continuity of $Q$ (Proposition \ref{P:PropertiesofQ}) and dominated convergence (using Assumption \ref{A:Assumptions2}(iii) to dominate the integrand by an integrable random variable), the right-hand side equals zero. 
The matching lower bound
$$
\liminf_{h \to 0} \bigl(\hat{v}(x+h)-\hat{v} (x)\bigr)\geq 0
$$
follows by the symmetric argument, using \eqref{eq:assumptions}.

For the growth bound on $v$,
$$
|v(x)| \leq \vert H(x) \vert + \delta^{-1}\E\left[\vert Q ({S}+x)\vert\indicator{{S}+x \geq x^\ast}\right].
$$
It therefore suffices to observe that if a continuous function $F$ satisfies
\begin{equation*}
      \vert F(x)\vert \leq K_1( 1+\cosh (\theta x))\quad \text{for all } x \in \xR, \text{ with } K_1>0, 
\end{equation*}
then 
\begin{equation*}
|\E_x F(S)|+ |\E_x F(I)| \leq K_2(1+\cosh(\theta x))\quad \text{for some } K_2>0.
\end{equation*}
This follows from the inequality
\begin{equation*}
    \vert \E_x F(S) \vert \leq K_1+\frac{K_1}{2}\bigl( \E e^{\theta x +\theta S }+\E e^{-\theta x- \theta S }  \bigr) \leq C (1+ \cosh (\theta x)),
\end{equation*}
with $C= K_1+\frac{K_1}{2}\max (\E e^{\theta S}, \E e^{-\theta S})$, which is finite by Assumption \ref{A:AssumptionsProcess}; the same argument applies with $I$ in place of $S$.
\end{proof}
The next regularity result follows from equation~(17) in \cite[Theorem~2]{MO-ECP}, observing that $P(S=0)=0$ when $X$ has unbounded variation.
\begin{prpstn}\label{P:SMOOTHFIT}
If $X$ has unbounded variation, then $v \in C^1(\xR)$.
\end{prpstn}
We will need the following auxiliary lemma, ensuring that the exponential growth bound is preserved under integration.
\begin{lmm}\label{P:primitivebounded}
Let $\theta>0$ and suppose that $v\colon\xR\to\xR$ satisfies
\[
|v(x)| \le K_1\bigl(1+\cosh(\theta x)\bigr)
\quad \text{for all } x\in\xR,
\]
for some constant $K_1>0$. Let $W$ be any primitive of $v$, i.e., $W'(x)=v(x)$.
Then there exists a constant $K_2>0$ such that
\[
|W(x)| \le K_2\bigl(1+\cosh(\theta x)\bigr)
\quad \text{for all } x\in\xR.
\]
\end{lmm}
The proof is a routine application of the fundamental theorem of calculus together with elementary inequalities for hyperbolic functions, and is omitted.
\begin{lmm}\label{L:Primitive}
Let $W$ be a primitive of $v$. Then $W$ belongs to the domain of the infinitesimal generator, and the map
\begin{equation*}
x\ \mapsto \mathcal{L}W(x)-\delta W(x) +f(x)
\end{equation*}
is constant on $(-\infty,x^\ast)$ and non-decreasing on $(x^\ast,\infty)$.
\end{lmm}
\begin{proof}
The regularity of $v$ established in Propositions \ref{P:FIT} and \ref{P:SMOOTHFIT}, together with Lemma \ref{P:primitivebounded} applied to $W$, allow us to apply Proposition \ref{P:INFINITESIMALGENERATOR}.
The proof then follows the same arguments as in \cite[Lemma~1(i)--(ii)]{MO}, where a Dynkin game is considered. In our setting there is only one barrier instead of two, and the same reasoning applies by taking (in the notation of that paper) $-\infty$ in place of $a^\ast_{\epsilon}$.
\end{proof}
From the previous results, we can prove the main theorem of the article.
\begin{proof}[Proof of Theorem \ref{T:DISCOUNTEDPROBLEMSOLUTION}]
By Lemmas \ref{L:optimalstopping}, \ref{P:primitivebounded}, and \ref{L:Primitive}, together with Theorem \ref{T:Verification}, it suffices to construct a primitive of $v$ for which the map $x \mapsto \mathcal{L}W(x) - \delta W(x) + f(x)$ vanishes on $(-\infty, x^\ast]$ (which, by Lemma \ref{L:Primitive}, reduces to vanishing at the single point $x^\ast$, since this map is constant on that interval). 
Consider 
$$W(x):= \int_{x^\ast}^x v(y)\,dy,$$
so that $W(x^\ast) = 0$, and let
$$
u(x)= W(x) +\frac{f(x^\ast) +\mathcal{L}W(x^\ast)}{\delta}. 
$$
Since $\mathcal{L}$ annihilates constants, $\mathcal{L}u = \mathcal{L}W$, and using $W(x^\ast) = 0$, we obtain
$$
\mathcal{L}u(x^\ast)-\delta u(x^\ast) + f(x^\ast) = \mathcal{L}W(x^\ast) - \bigl(f(x^\ast)+\mathcal{L}W(x^\ast)\bigr) + f(x^\ast) = 0.
$$
Moreover, $u' = v \leq c$ holds everywhere with equality on $[x^\ast,\infty)$ by Lemma \ref{L:optimalstopping}, and the growth bound $|u(x)| \leq K(1+\cosh(\theta x))$ follows from Lemma \ref{P:primitivebounded} applied to $W$. Thus $u$ satisfies all the hypotheses of Theorem \ref{T:Verification}, and the conclusion follows.
\end{proof}
%%%%%%%%%%%%%%%%%%%%%%%%%%%%%%%%%%%%%%%%%%%%%%%%%%%%%%%%%%%%%%%%%%%%%%%%%%%%%
%%%%%%%%%%%%%%%%%%%%%%%%%%%%%%%%%%%%%%%%%%%%%%%%%%%%%%%%%%%%%%%%%%%%%%%%%%%%%
%%%%%%%%%%%%%%%%%%%%%%%%%%%%%%%%%%%%%%%%%%%%%%%%%%%%%%%%%%%%%%%%%%%%%%%%%%%%%
\section{Applications and examples}
In this section, we introduce a motivating application and some examples in which the threshold $x^\ast$ can be obtained explicitly.
\begin{xmpl}[Application: Optimization in pollution abatement]%\label{SS:pollution}
We model a company that seeks to optimally adjust its pollution levels (or
abatement efforts) over time. Let $Z_{t}>0$ denote the current emission rate
(the instantaneous flow of $CO_{2}$ emissions). We assume that 
$Z_{t}=e^{X_{t}}$, where $X=\{X_{t}\colon t\geq 0\}$ is a general L\'evy process as defined in Section \ref{subsection:preliminaries}. 
This logarithm specification ensures the positivity of $Z_{t}$ and captures
multiplicative shocks: a negative jump in $X_{t}$ corresponds to a
proportional reduction in emission rate (for instance, due to improved
technology or temporary efficiency gains), while a positive jump corresponds
to a sudden increase in emissions, such as those arising from production
surges.

In this setting, the regulator (or company) can control the emission rate
through abatement actions. The cumulative log-scale reduction in the emission rate is encoded by the control process $D=\{D_t\colon t\geq 0\}$, so that the controlled state
$$X_{t}^{D}=X_{t}-D_{t}$$
yields the controlled emission rate
\begin{equation*}
Z_{t}^{D}=e^{X_{t}^{D}}= Z_t \, e^{-D_t}.
\end{equation*}
Without control ($D \equiv 0$), $Z_t$ follows its natural exponential L\'evy dynamics, representing the firm's spontaneous emission rate. The factor $e^{-D_t}$ thus quantifies the multiplicative emission reduction achieved through abatement up to time $t$.

The process $D$ is non-decreasing and right-continuous by Definition \ref{D:Admissiblestrategies}: increases in $D_{t}$
correspond to extra abatement efforts, which may correspond, for instance, to investments in clean technologies or to regulatory mandates. Hence the controlled state process reflects the combination of natural emission dynamics and deliberate abatement interventions.
Since $Z_{t}^{D}=e^{X_{t}^{D}}$, lower values of $X_{t}^{D}$
correspond to lower emissions; thus increases in $D_t$ reduce both $X_t^D$ and the emission rate $Z_t^D$.

The cost structure of the problem involves two functions. The \emph{running cost} $F(Z_t^D)$ is associated with the emission rate $Z_t^D$, capturing long-run economic or social costs such as environmental damage and regulatory penalties. A natural assumption is that $F$ is non-negative and non-decreasing in $Z_t^D$, since higher emissions entail higher environmental costs. The \emph{marginal abatement cost} $C(Z_t^D)$ represents the cost of further reducing emissions; it is typically non-negative and non-increasing in $Z_t^D$, reflecting that abatement is cheap when emissions are high---effective and inexpensive measures such as improving energy efficiency or switching to cleaner fuels are available---but becomes prohibitively expensive once emissions are already low, where the remaining reductions require radical technological or behavioral changes.
%%%%%%%%%%%%%%%%%%%%%%%%%%%%%%%%%%%%%%%%%%%%%%%%%%%%%%%%%%%%%%%%%%%%%%%%%%%%%
%%%%%%%%%%%%%%%%%%%%%%%%%%%%%%%%%%%%%%%%%%%%%%%%%%%%%%%%%%%%%%%%%%%%%%%%%%%%%
%%%%%%%%%%%%%%%%%%%%%%%%%%%%%%%%%%%%%%%%%%%%%%%%%%%%%%%%%%%%%%%%%%%%%%%%%%%%%

A natural and tractable particular case is the \emph{exponential cost
model}, obtained when 
$$
F(z) = Az^{\alpha}, \qquad C(z) = Bz^{-\beta},
\quad A,B,\alpha,\beta > 0.
$$
Passing to the log scale via $f(x)=F(e^x)$ and $c(x)=C(e^x)$, we obtain a running cost $f(X_t^D) = Ae^{\alpha X_t^D}$ that is non-decreasing in $X_t^D$, and a marginal abatement cost $c(X_t^D) = Be^{-\beta X_t^D}$ that is non-increasing in $X_t^D$, consistent with the economic interpretation
above. Assumption \ref{A:AssumptionsProcess} then requires that there exist $\theta \geq \max(\alpha,\beta)$ in the domain of $\phi$ with $\max(\phi(\theta), \phi(-\theta)) < \delta$.
Under these conditions, Theorem \ref{T:DISCOUNTEDPROBLEMSOLUTION} applies, and the optimal policy is a reflecting control at a critical barrier $x^{\ast}$. The corresponding emission threshold $z^\ast := e^{x^\ast}$ has a natural interpretation: it is the highest emission level tolerated under the optimal policy, beyond which abatement is triggered.
\begin{itemize}
\item If $X_{t}^{D}>x^{\ast}$, the optimal strategy is to increase $D_{t}$
(i.e., to abate emissions) so that the controlled state is pushed back to $%
x^{\ast }$ immediately.
\item If $X_{t}^{D}<x^{\ast }$, the optimal strategy is not to intervene, allowing the emission dynamics to evolve freely
until $X_{t}^{D}$ reaches $x^{\ast }$. The intuition is that when the
system is already in a ``clean'' state,
additional abatement yields little benefit while generating unnecessary
costs; no further abatement is undertaken until the state rises back to $x^{\ast}$.
\end{itemize}
We now derive $x^*$ explicitly. Since
$\mathcal{L}e^{\eta x} = \phi(\eta)e^{\eta x}$ for any $\eta$ in the domain of $\phi$, 
the function $r$
defined in \eqref{eq:r} satisfies
$$
r(x) = (\delta - \phi(-\beta))\,B\,e^{-\beta x},
$$
so that
$$
(f'-r)(x) = A\alpha\,e^{\alpha x}
  - (\delta-\phi(-\beta))\,B\,e^{-\beta x}.
$$
The Wiener--Hopf factorization at $-\beta$ gives the identity $\frac{\delta-\phi(-\beta)}{\delta}\,\E(e^{-\beta I}) = 1/\E(e^{-\beta S})$. Using this in $\E_x[(f'-r)(I)] = \E[(f'-r)(I+x)]$, the map \eqref{eq:CONDITION} becomes
$$
Q(x) = \E_x\bigl[(f'-r)(I)\bigr]
= K_1\,e^{\alpha x} - K_2\,e^{-\beta x},
$$
where 
$$
K_1 = A\alpha\,\E(e^{\alpha I}), \qquad
K_2 = \frac{\delta B}{\E(e^{-\beta S})}.
$$
Since $K_1, K_2 > 0$ and $Q'(x) = \alpha K_1 e^{\alpha x} + \beta K_2 e^{-\beta x} > 0$
for all $x$, the map $Q$ is strictly increasing, so
Assumption \ref{A:Assumptions4} is automatically satisfied and the
root $x^*$ is unique. Setting $Q(x^*) = 0$ gives
$K_1 e^{(\alpha+\beta)x^*} = K_2$, whence 
\begin{equation}\label{eq:xstar}
x^* = \frac{1}{\alpha+\beta}
  \log\frac{\delta B}{A\alpha\,\E(e^{\alpha I})\,\E(e^{-\beta S})}\geq 
  \frac{1}{\alpha+\beta}
  \log\frac{ \delta B}{A\alpha},
\end{equation}
where the inequality follows from $\E(e^{\alpha I}) \leq 1$ and $\E(e^{-\beta S}) \leq 1$, one of them strict for a non-trivial L\'evy process.

This expression yields the following insights:
\begin{itemize}
     \item $z^\ast=e^{x^*}\geq (\delta B/(A\alpha))^{1/(\alpha+\beta)}$ is bounded away from zero, reflecting the unbounded growth of the marginal abatement cost as $z \to 0^+$, which makes pushing emissions arbitrarily close to zero prohibitively expensive.
     \item The dependence of ${z^\ast}$ on $A$ and $B$ is only through the ratio $B/A$: rescaling both costs by a common factor leaves the optimal threshold unchanged.
    \item As $\alpha \to 0$, $x^\ast \to \infty$ (and hence $z^\ast = e^{x^\ast} \to \infty$). This matches the intuition that when the running cost becomes nearly constant, abatement provides no benefit and the optimal policy is to never intervene.
\end{itemize}

The formula \eqref{eq:xstar} for $x^*$ holds for any L\'evy process satisfying the standing assumptions.
To make it explicit, consider the two-sided jump-diffusion process
\begin{equation}\label{eq:CPPmodel2}
X_t = b\,t + \sigma W_t
  + \sum_{i=1}^{N^{(1)}_t} Y^{(1)}_i
  - \sum_{i=1}^{N^{(2)}_t} Y^{(2)}_i,
\end{equation}
where $W = \{W_t\}$ is a standard Brownian motion; for $k=1,2$, the
processes $N^{(k)} = \{N^{(k)}_t\}$ are independent Poisson processes with
intensities $\lambda_k$; and $\{Y^{(k)}_i\}_{i \geq 1}$ is an i.i.d.\ sequence of exponential random variables with parameter $\eta_k$. The five families $W$, $N^{(1)}$, $N^{(2)}$, $Y^{(1)}$, $Y^{(2)}$ are mutually independent. 
The characteristic exponent of $X$ is
$$
\phi(z) = \tfrac{1}{2}\sigma^2 z^2 + b\,z  
  + \frac{\lambda_1 z}{\eta_1 - z}
  - \frac{\lambda_2 z}{\eta_2 + z}.
$$
The equation $\phi(z) = \delta$ has four real roots $\rho_1, \rho_2 >
0$ and $-\gamma_1, -\gamma_2 < 0$ satisfying
$$
-\gamma_2 < -\eta_2 < -\gamma_1 < 0
  < \rho_1 < \eta_1 < \rho_2.
$$
By \cite[Theorem~3.1]{Firstpassage}, the Wiener--Hopf factors of $X$
killed at rate $\delta$ are
\begin{equation}\label{eq:WHfactors}
\E(e^{z S})
  = \frac{\rho_1\rho_2(\eta_1 - z)}
         {\eta_1(\rho_1 - z)(\rho_2 - z)},
\qquad
\E(e^{z I})
  = \frac{\gamma_1\gamma_2(\eta_2 + z)}
         {\eta_2(\gamma_1 + z)(\gamma_2 + z)}.
\end{equation}
Substituting into \eqref{eq:xstar} yields 
\begin{equation*}
x^* = \frac{1}{\alpha+\beta}\left[
  \log\!\left(\frac{\delta B}{A\alpha}\right)
  + \log\!\left(
    \frac{\eta_1\eta_2(\gamma_1+\alpha)(\gamma_2+\alpha)(\rho_1+\beta)(\rho_2+\beta)}
         {\gamma_1\gamma_2\rho_1\rho_2(\eta_2+\alpha)(\eta_1+\beta)}\right)
  \right].
\end{equation*}
In this example, Assumption \ref{A:AssumptionsProcess} holds for $\theta = \max(\alpha,\beta)$ iff $\max(\phi(\theta), \phi(-\theta)) < \delta$, which by the root structure above is equivalent to $\max(\alpha,\beta) < \min(\gamma_1,\rho_1)$.
%
%\subsection{Example: linear controls, monomial running cost}
%
\end{xmpl}

\begin{xmpl}[Linear controls, monomial running cost]
Let $f(x)=Ax^{2n}+B$ and $c(x)=-x$, with $A>0$, $B \in \xR$, and $n \in \xN$. 
Since $f$ and $c$ are polynomials, they are dominated by $K(1+\cosh(\theta x))$ for any $\theta > 0$, so Assumption \ref{A:Assumptions2}(iii) holds as soon as the process satisfies Assumption \ref{A:AssumptionsProcess} for some $\theta > 0$. In this case $r(x)= \delta(\E X_{e_\delta} - x)$ and the function
$$
(f'-r)(x)= 2An x^{2n-1} - \delta(\E X_{e_{\delta}}-x) = 2An x^{2n-1} +\delta x - \delta \E X_{e_\delta}
$$
is strictly increasing, so it admits a unique root.

In the particular case $n=1$, using the Wiener--Hopf identity $\E X_{e_\delta} = \E S + \E I$, 
\begin{equation*}
    Q(x) = \E_x[(f'-r)(I)] = (2A+\delta)(x+\E I) -\delta \E X_{e_\delta} 
    =  (2A+\delta) x+2A \E I -\delta \E S , 
\end{equation*}
whence 
$$
x^\ast= \frac{\delta \E S - 2A\E I}{2A+\delta}.
$$
As an illustration, consider the process defined in \eqref{eq:CPPmodel2}. From \eqref{eq:WHfactors}, $\E S = \rho_1^{-1} + \rho_2^{-1} - \eta_1^{-1}$ and $\E I = \eta_2^{-1} - \gamma_1^{-1} - \gamma_2^{-1}$, yielding 
$$
x^\ast = \frac{\delta(\rho_1^{-1}+\rho_2^{-1}-\eta_1^{-1} )-2A(\eta_2^{-1}-\gamma_1^{-1}-\gamma_2^{-1} )}{2A+\delta}.
$$
The same argument applies more generally: for any convex running cost $f$ satisfying the regularity and integrability conditions of Assumptions \ref{A:Assumptions2} and \ref{A:Assumptions3Unboundedvariation}, with $f'(x) \to \pm \infty$ as $x \to \pm\infty$, there exists a unique root $x^\ast$.
\end{xmpl}
%
%\subsection{Example: exponential running cost and quadratic control cost}
%
\begin{xmpl}[Exponential running cost and quadratic control cost]
Let $f(x)=e^{x}$ and $c(x)=Ax^2+B$, with $A>0$ and $B\in \xR$. 
In this example 
$$
r(x)= \delta(Ax^2+B)-A\left(2x\E X_1+\var X_1\right),
$$
thus
$$
(f'-r)(x)=  e^x-\delta Ax^2- \delta B +A\left(2x\E X_1+\var X_1\right), 
$$
then
$$
Q(x)=\E(e^I)e^x-\delta A\E(x+I)^2-\delta B+A\left(2\E(x+I)\E X_1+\var X_1\right).
$$
To study the monotonicity of $Q(x)$ we differentiate 
\begin{equation*}
Q'(x)=\E(e^I)e^x-\delta 2A\E(x+I)+2A\E X_1
\geq e^{x+\E I}-\delta 2A\E(x+I)+2A \E X_1> 0,
\end{equation*}
provided $\delta 2A\leq e$ and $\E X_1> 0$ (where we used Jensen's inequality).
Hence, $Q$ is strictly increasing and has exactly one root, satisfying Assumption \ref{A:Assumptions4}. 

As an illustration, consider a compound Poisson process $X=\{X_t\colon t\geq 0\}$ with double-sided exponential jumps, given by
\begin{equation*}
X_t=x+\sum_{i=1}^{N^{(1)}_t}Y^{(1)}_i-\sum_{i=1}^{N^{(2)}_t}Y^{(2)}_i,
\end{equation*}
where $N^{(1)}=\{N^{(1)}_t\colon t\geq 0\}$ and 
$N^{(2)}=\{N^{(2)}_t\colon t\geq 0\}$ 
are two Poisson processes with positive intensities $\lambda_1,\lambda_2$ respectively;
$Y^{(1)}=\{Y^{(1)}_i\colon i\geq 1\}$ and $Y^{(2)}=\{Y^{(2)}_i\colon i\geq 1\}$ 
are two sequences of independent exponentially distributed random variables with respective positive parameters 
$\alpha_1,\alpha_2$.
The four families $N^{(1)}$, $N^{(2)}$, $Y^{(1)}$, $Y^{(2)}$ are mutually independent. Again, using \cite[Theorem~3.1]{Firstpassage}, the equation $\phi(z) = \delta$ has two roots $\gamma_1,\gamma_2$ (here, with $\gamma_1$ a positive root and $-\gamma_2$ a negative root) that satisfy $-\alpha_2<-\gamma_2<0<\gamma_1<\alpha_1$, and 
\begin{align*}
f_I(x)&=\frac{\gamma_2}{\alpha_2}\delta_0(x)+\frac{\alpha_2-\gamma_2}{\alpha_2}\gamma_2e^{\gamma_2x},\quad x\leq 0,\\
f_S(x)&=\frac{\gamma_1}{\alpha_1}\delta_0(x)+\frac{\alpha_1-\gamma_1}{\alpha_1}\gamma_1 
e^{-\gamma_1x},\quad x\geq 0,
\end{align*}
where $\delta_0$ denotes the Dirac mass at $x=0$. A direct computation yields

\begin{multline*}
Q(x)
=
\frac{\gamma_2(\alpha_2+1)}
{\alpha_2(1+\gamma_2)}e^x-\delta B\\
+A\Bigg[
-\delta x^2
+2\left(
\frac{\lambda_1}{\alpha_1}
-\frac{\lambda_2}{\alpha_2}
-\delta\frac{\gamma_2-\alpha_2}{\alpha_2\gamma_2}
\right)x\\
+\frac{2\left(\frac{\lambda_1}{\alpha_1}
-\frac{\lambda_2}{\alpha_2}\right)
(\gamma_2-\alpha_2)}
{\alpha_2\gamma_2}
-\frac{2\delta(\alpha_2-\gamma_2)}
{\alpha_2\gamma_2^2}
+\frac{2\lambda_2}{\alpha_2^2}
+\frac{2\lambda_1}{\alpha_1^2}
\Bigg].
\end{multline*}

In this setting, the only delicate requirement is Assumption \ref{A:AssumptionsProcess}, which holds iff there exists $\theta \geq 1$ such that 
$$  
\max \left(\frac{\theta \lambda_1 }{\alpha_1-\theta}- \frac{\theta \lambda_2 }{\alpha_2 +\theta}, \ \frac{\theta \lambda_2 }{\alpha_2-\theta}- \frac{\theta \lambda_1 }{\alpha_1 +\theta}\right) < \delta.
$$
For instance, choosing
$$
\lambda_1=4/7, \ \lambda_2=3/7,  \ \alpha_1=3,  \ \alpha_2=4,\ A=B=\delta=\theta=1, 
$$ 
we obtain $x^{\ast} \approx -0.0377$.
\end{xmpl}

\end{document}